\documentclass[oneside]{amsart}

\usepackage[letterpaper,body={14.6cm,22.7cm}, mag=1000]{geometry}
\usepackage{amssymb}
\usepackage{amsthm}
\usepackage{amscd}

\numberwithin{equation}{section}
\theoremstyle{plain}
\newtheorem{cor}[equation]{Corollary}
\newtheorem{lemma}[equation]{Lemma}

\newtheorem{prop}[equation]{Proposition}

\newtheorem{thm}[equation]{Theorem}
\newtheorem*{hypa}{Hypothesis A}
\newtheorem*{thma}{Theorem A}
\newtheorem*{thmb}{Theorem B}
\newtheorem*{thmc}{Theorem C}
\theoremstyle{definition}

\newtheorem{remark}[equation]{Remark}

\newcommand{\dlabel}[1]{\ifmmode \text{\ttfamily \upshape [#1] } \else
{\ttfamily \upshape [#1] }\fi \label{#1}}
\newcommand{\A}{\operatorname{A} }
\newcommand{\B}{\operatorname{B} }
\newcommand{\C}{\operatorname{C} }
\newcommand{\D}{\operatorname{D} }
\newcommand{\E}{\operatorname{E} }
\newcommand{\F}{\operatorname{F} }
\newcommand{\K}{\operatorname{K} }

\newcommand{\M}{\operatorname{M} }
\newcommand{\N}{\operatorname{N} }
\newcommand{\Ha}{\operatorname{H} }

\newcommand{\Z}{\operatorname{Z} }

\newcommand{\gen}[1]{\left < #1 \right >}
\newcommand{\Aut}{\operatorname{Aut} }

\newcommand{\Hom}{\operatorname{Hom} }
\newcommand{\Inn}{\operatorname{Inn} }

\newcommand{\Autcent}{\operatorname{Autcent} }

\begin{document}

\title{On finite $p$-groups whose central automorphisms are all class preserving}

\author{Manoj K.~Yadav}

\address{School of Mathematics, Harish-Chandra Research Institute \\
Chhatnag Road, Jhunsi, Allahabad - 211 019, INDIA}

\email{myadav@mri.ernet.in}
\subjclass[2000]{Primary 20D45; Secondary 20D15}
\keywords{finite $p$-group; class preserving automorphism; central automorphism}
\date{\today}

\begin{abstract}
We obtain certain results on a finite $p$-group whose central automorphisms are all class preserving. In particular, we prove that if $G$ is a finite $p$-group whose central automorphisms are all class preserving, then $d(G)$ is even, where $d(G)$ denotes the number of elements in any minimal generating set for $G$. As an application of these results, we obtain some results regarding finite $p$-groups whose automorphisms are all class preserving.
\end{abstract}
\maketitle

\section{Introduction}
Let $G$ be any group. By $\Z(G)$, $\gamma_2(G)$ and $\Phi(G)$,  we denote the center, the commutator subgroup and
the Frattini subgroup of $G$ respectively.  For $x \in G$, $[x, G]$ denotes the set $\{[x, g] =  x^{-1}g^{-1}xg | g \in G\}$ and $x^G$ denotes the conjugacy class of 
$x$ in $G$. Notice that $x^G = x[x,G]$ and therefore $|x^G| = |[x,G]|$ for all $x \in G$.  For $x \in G$, $\C_{H}(x)$ denotes the 
centralizer of $x$ in $H$, where $H$ is a subgroup of $G$. To say that some $H$ is a subgroup (proper subgroup) of $G$ we write 
$H \leq G$ ($H < G$). For  any group $H$ and an abelian group $K$,  $\Hom(H, K)$ denotes the group of all homomorphisms from $H$ to $K$.
For a finite $p$-group $G$, we denote by $\Omega_m(G)$ the subgroup $\gen{x \in G \mid x^{p^m} = 1}$ and by $\Omega^m$ (which is not a standard notation) the subgroup $\gen{x^{p^m} \mid x \in G}$, where $p$ is a prime integer and $m$ is a positive integer.
Let $d(G)$ denote the number of elements in any minimal generating set for  a finite $p$-group $G$. 

An automorphism $\phi$ of a group $G$ is called  \emph{central} if 
$g^{-1}\phi(g) \in \Z(G)$ for all  $g \in G$. The set of all central 
automorphisms of $G$, denoted by $\Autcent(G)$, is a normal subgroup 
of $\Aut(G)$. Notice that $\Autcent(G) = \C_{\Aut(G)}(\Inn(G))$, where $\Inn(G)$ denotes the group of all inner automorphisms of $G$. 
An automorphism $\alpha$ of $G$ is called \emph{class preserving} if 
$\alpha(x) \in x^G$ for all $x \in G$. The set of all class preserving
automorphisms of $G$,  denoted by $\Aut_{c}(G)$, is a normal subgroup of 
$\Aut(G)$. Notice that $\Inn(G)$ is a normal subgroup of $\Aut_c(G)$.

 In 1999, A.  Mann \cite[Question 10]{aM99} asked the following question: \textit{Do all $p$-groups have automorphisms that are not class preserving? If the answer is no, which are the groups that have only class preserving automorphisms?} The first part of the question have a negative answer. The examples of finite $p$-groups $G$ such that $\Aut(G) = \Aut_c(G)$ are already known in the literature. Such groups $G$, having nilpotency class $2$ were constructed by H. Heineken \cite{hH80} in 1980 and that having nilpotency class $3$ were constructed by  I. Malinowska \cite{iM92} in 1992. So the second part of the question of Mann becomes relevant. Let us modify the question of Mann to make it more precise in the present scenario.  \\

\noindent{\bf Question.} Let $n \ge 4$ be a positive integer and $p$ be a prime number. Does there exist a finite $p$-group of nilpotency class $n$ such that $\Aut(G) = \Aut_c(G)$.\\

The second part of Mann's question, which clearly talks about the classification, can be stated as\\

\noindent{\bf Problem (A. Mann).} Study (Classify) finite $p$-groups $G$ such that $\Aut(G) = \Aut_c(G)$.\\

In this paper we consider this problem for $p$-groups of nilpotency class $2$ and  make the stone rolling. 
Notice that, for a finite $p$-group of class  $2$, we have the following sequence of subgroups
\[\Aut_c(G ) \le \Autcent(G) \le \Aut(G).\]
The groups $G$ such that $\Autcent(G) = \Aut(G)$, have been studied, every now and then, by many mathematicians (see \cite{JY12} and \cite{JRY} for recent developments and other references).  So for studying groups $G$ of nilpotency class $2$ such that $\Aut(G) = \Aut_c(G)$, one needs to concentrate on the groups $G$ satisfying 

\begin{hypa}\label{hypa}
$\Aut_c(G) = \Autcent(G)$.
\end{hypa}

Suppose that  $\Aut_c(G) = \Autcent(G)$, for some finite group $G$. Then $\Inn(G) \le \Autcent(G)$. It then follows from the definition of $\Autcent(G)$ that $\Inn(G)$ is abelian. Hence the 
nilpotency class of $G$ is at the most $2$.   Since a finite nilpotent group $G$ can be written as a direct product of its Sylow $p$-subgroups, where $p$ is a prime,  to study $\Autcent(G)$, it is sufficient to study the group of central 
automorphisms of finite $p$-group for the relevant prime integers $p$. Suppose that $G$ is abelian and satisfies Hypothesis A, then  $\Aut_c(G) = 1$ and $\Autcent(G) =
\Aut(G)$. Thus $\Aut(G) = 1$. But this is possible only when  $|G| \le 2$. So from now onwards, we concentrate on finite $p$-groups of class $2$, where $p$ is a prime integer.

Let $G$ be a finite $p$-group of class $2$. Then $G/\Z(G)$ is abelian. Consider the following cyclic decomposition of $G/\Z(G)$.
\[G/\Z(G) = \C_{p^{m_1}}  \times \cdots \times \C_{p^{m_d}}\]
such that $m_1 \ge m_2 \ge \cdots \ge m_d \ge 1$, where $\C_{p^{m_{i}}}$ denotes the cyclic group of order $p^{m_i}$ for $1 \le i \le d$. The integers $p^{m_1}, \ldots, p^{m_d}$ are unique for $G/\Z(G)$ and these are called the \emph{invariants} of $G/\Z(G)$. Now we state our first result in the following theorem, which we prove in Section 3 as Theorem \ref{thm1}.

\begin{thma}
Let $G$ be a finite $p$-group of class $2$ and  $p^{m_1}, \ldots, p^{m_d}$ be the invariants of  $G/\Z(G)$. Then $G$ satisfies Hypothesis A if and only if $\gamma_2(G) = \Z(G)$ and $|\Aut_c(G)| = \Pi_{i=1}^d |\Omega_{m_i}(\gamma_2(G))|$.  
\end{thma}

Our next result is the following theorem, which we prove in Section 3 as Theorem \ref{thm2}.

\begin{thmb}
Let $G$ be a  finite $p$-group of class $2$ satisfying Hypothesis A. Then $d(G)$ is even.
\end{thmb}

In the last section we concentrate on finite $p$-groups whose automorphisms are all class preserving and prove the following result.

\begin{thmc}
Let $G$ be a non-abelian finite $p$-group such that $\Aut(G) = \Aut_c(G)$, where $p$ is an odd prime. Then the following statements hold true.
\begin{subequations}
\begin{align}
& \text{$\gamma_2(G)$ cannot be cyclic.}\\
&\text{If $\Aut(G)$ is elementary abelian, then $G$ is a Camina special $p$-group.}\\ 
&\text{If $\Aut(G)$ is abelian, then $d(G)$ is even.}\\
&\text{If $\Aut(G)$ is abelian,   then $|G| \ge p^8$ and $|\Aut(G)| \ge p^{12}$.}\\
&\text{With $\Aut(G)$ abelian, $|\Aut(G)| = p^{12}$ if and only if $|G| = p^8$.}\\
&\text{If $\Aut(G)$ is abelian of order $p^{12}$, then $\Aut(G)$ is elementary abelian}\\
&\text{There exists a group $G$ of order $3^8$ such that $|\Aut(G)| = |\Aut_c(G)| = 3^{12}$.}\label{goodeq7}
\end{align}
\end{subequations}
\end{thmc}

In Section 2, we collect some basic results, which are useful for our work.  Some further properties and some examples of finite groups satisfying Hypothesis A are obtained in Section 4.

We conclude this section with some definitions.
A subset $\{y_1, \ldots, y_d\}$ of a finite abelian group $Y$ is said to be a \emph{minimal basis} for $Y$ if 
\[Y =\gen{y_1} \times \gen{y_2} \times \ldots \times \gen{y_d} ~\textrm{ and}~ |\gen{y_1}| \ge |\gen{y_2}| \ge \cdots \ge |\gen{y_d}| > 1.\] 
A minimal generating set  $\{x_1, \ldots, x_d\}$ of a finite $p$-group $G$ of nilpotency class $2$ is said to be  \emph{distinguished} if  the set $\{{\bar x}_1, \ldots, {\bar x}_d\}$, ${\bar x}_i = x_i\Z(G)$, forms a  minimal basis for $G/\Z(G)$.


\section{Some prerequisites and useful lemmas}

Let $G$ be a finite group.
Let $\alpha \in \Autcent(G)$. Then the map $f_{\alpha}$ from $G$ into $\Z(G)$
defined by $f_{\alpha}(x) = x^{-1}\alpha(x)$ is a homomorphism which sends
$\gamma_2(G)$ to $1$. Thus $f_{\alpha}$ induces a homomorphism from
$G/\gamma_2(G)$ into $\Z(G)$. So we get a one-to-one map $\alpha \rightarrow
f_{\alpha}$ from $\Autcent(G)$ into $\Hom(G/\gamma_2(G), \Z(G))$. Conversely, 
if $f \in \Hom(G/\gamma_2(G), \Z(G))$, then $\alpha_f$ such that
$\alpha_f(x) = xf({\bar x})$ defines an endomorphism of $G$, where 
${\bar x} = x \gamma_2(G)$ . But this, in general, may not be an automorphism of $G$. More precisely, $\alpha_f$ fails to be an automorphism of $G$ when $G$ admits a non-trivial abelian direct factor.

A group $G$ is called \emph{purely non-abelian} if it does not have a  non-trivial abelian direct factor.

The following theorem of Adney and Yen \cite{AY65} shows that if $G$ is a purely non-abelian finite group, then the mapping   $\alpha \rightarrow
f_{\alpha}$ from $\Autcent(G)$ into $\Hom(G/\gamma_2(G), \Z(G))$, defined above,  is also onto.

\begin{thm}[\cite{AY65}, Theorem 1]\label{lemma0}
Let $G$ be a purely non-abelian finite group. Then the correspondence $\alpha
\rightarrow f_{\alpha}$ defined above is a one-to-one mapping of $\Autcent(G)$
onto $\Hom(G/ \gamma_2(G),$ $ \Z(G))$.
\end{thm}

The following lemma  follows from  \cite[page 141]{AY65}.

\begin{lemma}\label{lemma1}
Let $G$ be a finite $p$-group of class $2$ such that $\Z(G) = \gamma_2(G)$.   Then $\Autcent(G)$ is abelian.
\end{lemma}

The following five lemmas, which we, sometimes,  may use without any further reference, are well known.

\begin{lemma}\label{lemma5}
Let $\C_n$ and $\C_m$ be two cyclic groups of order $n$ and $m$
respectively. Then $\Hom(\C_n, \C_m) \cong \C_d$, where $d$ is the greatest
common divisor of $n$ and $m$, and $\C_d$ is the cyclic group of order $d$.
\end{lemma} 

\begin{lemma}\label{lemma5a}
Let $\A$, $\B$ and $\C$ be finite abelian groups. Then

{\em (i)} $\Hom(\A \times \B, \C) \cong \Hom(\A,\C) \times \Hom(\B,\C)$;

{\em (ii)}  $\Hom(\A,  \B \times \C) \cong \Hom(\A,\B) \times \Hom(\A,\C)$.
\end{lemma}

\begin{lemma}\label{lemma5b}
Let $\A$, $\B$ and $\C$ be finite abelian groups such that $\A$ and $\B$ are
isomorphic. Then  $\Hom(\A , \C) \cong  \Hom(\B,\C)$.
\end{lemma}

\begin{lemma}\label{lemma5c}
Let $\A$ and $\C$ be finite abelian groups and $\B$ is a proper subgroup of
$\C$. Then $|\Hom(\A , \B)| \le |\Hom(\A , \C)|$.
\end{lemma} 

\begin{lemma}\label{lemma2.7}
Let $\C_{p^m}$ be a cyclic group of order $p^m$ and $B$ be any finite abelian group. Then $|\Hom(\C_{p^m}, \B)| = |\Hom(\C_{p^m} , \Omega_{m}(B))|$.
\end{lemma}

The following lemma seems well known. But we include  a proof here, because we 
could not find a suitable reference for it.

\begin{lemma}\label{lemma5d}
Let $G$ be a finite abelian $p$-group and $M$ be a maximal subgroup of $G$. 
Then there exists a subgroup $\Ha$ of $G$ and a positive integer $i$ such that 
$G = \Ha \times C_{p^{i+1}}$ and $\M = \Ha \times C_{p^i}$.
\end{lemma}
\begin{proof}
We prove the lemma  by induction on the finite order $|G| \ge p$ of $G$.  Notice  that $|G| = 1$ is
impossible because $G$ has a subgroup $M < G$.  The lemma holds
trivially when $|G| = p$. So we may assume that $|G| > p$, and that the
lemma holds for all strictly smaller values of $|G|$.

     Let $q = p^e$ be the exponent  of $G$. Notice that $q > 1$
since $G \neq 1$.  Any element $y \in G$ with order $q$ lies in a minimal basis
for $G$.  So there is some subgroup $I$ of $G$ such that $G = \gen{y} \times I$. Furthermore $|I| < |G|$ since $y$ has order $q >
1$.

     Suppose that $M$ contains the above element $y$ of order $q$.  Then
$M = \gen{y} \times (M \cap I)$, where $M \cap I$ is a
maximal subgroup of $I$. By induction there exist a subgroup $J < I$ and
an element $x \in I$ such that $I = J \times \left < x \right >$ and $M
\cap I = J \times \gen{ x^p}$. The lemma now holds with this
$x$ and the subgroup $H = J \times \gen{y}$ of $G$.

     We have  handled every case where $M$ contains some element
$y$ of order $q$.  So we may assume from now on that no such $y$ lies in
$M$.  Let $x_1, x_2, \dots, x_d$ be a minimal basis for $G$.  Then the order of $x_1$ is the exponent $q$ of $G$.
Since $M$ contains no element of order $q$ in $G$, it is contained in the subgroup
\[ K = \gen{x_1^p} \times  \gen{x_2}  \times \dots \times \gen{ x_d}, \]
which has index $p$ in $G$. Since $M$ is maximal in $G$, it must equal
$K$. Then the lemma holds with $x = x_1$ and $H = \gen{ x_2}  \times \dots \gen{x_d}$.   \hfill $\Box$
\end{proof}

The following interesting lemma is from \cite[Lemma D]{CM01}. 
\begin{lemma}\label{lemma6}
Let $\A$, $\B$, $\C$ and $\D$ be finite abelian $p$-groups such that $\A$ is isomorphic to a proper subgroup of $\B$ and $\C$ is isomorphic to a proper 
subgroup of $\D$. Then  $|\Hom(\A , \C)| <  |\Hom(\B,\D)|$.
\end{lemma}
\begin{proof}
By Lemma \ref{lemma5c}, it is sufficient to prove the result when $\A$ is 
isomorphic to a maximal subgroup $\M$ of $\B$ and $\C$ is isomorphic to a maximal subgroup $\N$ of $\D$. 
By  Lemma \ref{lemma5d}, we have $\M \cong \E \times \C_{p^i}$ 
and  $\B \cong \E \times \C_{p^{i+1}}$ for some group $\E$ and some non-negative
integer $i$. Also  $\N \cong \F \times \C_{p^j}$ and $\D \cong \F \times \C_{p^{j+1}}$ for some 
group $\F$ and some non-negative integer $j$.  By Lemma \ref{lemma5a}
\[\Hom(\M,\N) \cong \Hom(\E,\F) \times \Hom(\E,\C_{p^j}) \times
\Hom(\F,\C_{p^i}) \times \Hom(\C_{p^i},\C_{p^j})\]
and 
\[\Hom(\B,\D) \cong \Hom(\E,\F) \times \Hom(\E,\C_{p^{j+1}}) \times
\Hom(\F,\C_{p^{i+1}}) \times \Hom(\C_{p^{i+1}},\C_{p^{j+1}}).\]
Since  $|\Hom(\E,\C_{p^j})| \le |\Hom(\E,\C_{p^{j+1}})|$, $|\Hom(\F,\C_{p^i})|
\le |\Hom(\F,\C_{p^{i+1}})|$ and  $|\Hom(\C_{p^i},\C_{p^j})|$ $ < 
|\Hom(\C_{p^{i+1}},\C_{p^{j+1}})|$, it follow that  $|\Hom(\M,\N)| 
< |\Hom(\B,\D)|$. Now by Lemma \ref{lemma5b} we get $\Hom(\M,\N) \cong 
\Hom(\A , \C)$. Hence  $|\Hom(\A,\C)| < |\Hom(\B,\D)|$.  \hfill $\Box$  
\end{proof}

\section{Groups $G$ satisfying Hypothesis A}

In this section we derive some interesting properties of finite groups  satisfying Hypothesis A and prove Theorems A and B.  We start with the following easy lemma.

\begin{lemma}\label{lemma6a}
Let $G$ be a finite $p$-group of class $2$. Then the following holds:
\begin{enumerate}
\item The exponents of $\gamma_2(G)$ and $G/\Z(G)$ are same.
\item For each $x \in G-\Z(G)$, $[x,G]$ is a non-trivial normal subgroup of $G$ contained in 
$\gamma_2(G)$.
\item For $x \in G-\Z(G)$, the  exponent of the subgroup $[x,G]$ is equal to the order of 
${\bar x} = x\Z(G) \in G/\Z(G)$.
\end{enumerate}
\end{lemma}
\begin{proof}  Since (1) and (2) are well known, we only  prove (3). 
Let the order of ${\bar x} = x\Z(G)$ be $p^c$. Then $x^{p^c} \in \Z(G)$. Let
$[x,g] \in [x,G]$ be an arbitrary element. Now $[x,g]^{p^c} = [x^{p^c}, g] =
1$. Thus the exponent of $[x,G]$ is less than or equal to $p^c$. We claim that it
can not be less than $p^c$. Suppose that the exponent of $[x,G]$ is $p^b$, where
$b < c$. Then $[x^{p^b},g] = [x,g]^{p^b} = 1$ for all $g \in G$. This proves
that $x^{p^b} \in \Z(G)$, which gives a contradiction to the fact that order
of ${\bar x}$ is $p^c$. Hence exponent of $[x,G]$ is equal to $p^c$, which completes the proof of the lemma.
\hfill $\Box$
\end{proof}

Let $G$ be a finite nilpotent group of  class $2$. Let $\phi \in \Aut_c(G)$. Then the map $g \mapsto g^{-1}\phi(g)$ is a homomorphism of
$G$ into $\gamma_2(G)$. This homomorphism sends $Z(G)$ to $1$. So it induces a homomorphism $f_{\phi} \colon G/Z(G) \to \gamma_2(G)$, sending
$gZ(G)$ to $g^{-1}\phi(g)$, for any $g \in G$.  It can be easily seen that
the map $\phi \mapsto f_{\phi}$ is a monomorphism of the group
$\Aut_c(G)$ into $\Hom(G/Z(G), \gamma_2(G))$.

Any $\phi \in \Aut_c(G)$ sends any $g \in G$ to some $\phi(g)
\in g^G$. Then $f_{\phi}(gZ(G)) = g^{-1}\phi(g)$ lies in $g^{-1}g^G =
[g,G]$.  Denote
\[  \{ \, f \in \Hom(G/Z(G), \gamma_2(G)) \mid f(gZ(G)) \in [g,G], \text{ for
all $g \in G$}\,\} \]
by $\Hom_c(G/Z(G), \gamma_2(G))$. Then $f_{\phi} \in \Hom_c(G/Z(G),
\gamma_2(G))$ for all $\phi \in \Aut_c(G)$. On the other hand, if $f \in
\Hom_c(G/Z(G), \gamma_2(G))$, then the map sending any $g \in G$ to
$gf(gZ(G))$ is an automorphism $\phi \in \Aut_c(G)$ such that $f_{\phi}
= f$. Thus we have

\begin{prop}\label{prop1}
 Let $G$ be a finite nilpotent group of class 2. Then the
above map $\phi \mapsto f_{\phi}$ is an isomorphism of the group
$\Aut_c(G)$ onto $\Hom_c(G/Z(G), \gamma_2(G))$.
\end{prop}

We also need the following easy observation.

\begin{lemma}\label{lemma7}
Let $H = \C_{n_1} \times \cdots \times \C_{n_r}$ and $K = \C_{m_1} \times 
\cdots \times \C_{m_s}$ be two finite abelian groups. Let $r \le s$ and 
$n_i$ divides $m_i$, where $1 \le i \le r$. Then $H$ is isomorphic to a
subgroup of $K$.
\end{lemma}

We use above information and Lemma \ref{lemma6} to prove the following.

\begin{prop}\label{prop2}
Let $G$ be a finite $p$-group of class $2$ which satisfies Hypothesis A. Then
$\gamma_2(G) = \Z(G)$.
\end{prop}
\begin{proof}
We first prove that $G$ is purely non-abelian. Suppose the contrary, then
$G = \K \times \A$, where $\A$ is a non-trivial abelian subgroup of
$G$ and $K$ is purely non-abelian. Obviously $|\Aut_c(G)| = |\Aut_c(\K)|$. It follows that  $|\Autcent(G)| >
|\Autcent(K)||\Autcent(\A)| \ge |\Autcent(K)|$ (Notice that the second inequality is also strict if $p$ is odd, since $\A$ is non-trivial). Hence 
\[|\Aut_c(G)| = |\Aut_c(\K)| \le |\Autcent(K)| < |\Autcent(G)|.\] 
This is a contradiction to Hypothesis A. This proves that $G$ is purely
non-abelian. 

Now suppose that $\gamma_2(G) < \Z(G)$. Then $|G/\Z(G)| < |G/\gamma_2(G)|$. 
Notice that all the conditions of Lemma \ref{lemma7} hold with $H = G/\Z(G)$ and 
$K = G/\gamma_2(G)$. Thus $G/\Z(G)$ is isomorphic to a proper subgroup of 
$G/\gamma_2(G)$. It follows from Theorem \ref{lemma0} that $|\Autcent(G)| =
|\Hom(G/\gamma_2(G), \Z(G))|$, since $G$ is purely non-abelian. By  Proposition \ref{prop1}, we have 
$|\Aut_c(G)| \le |\Hom(G/\Z(G), \gamma_2(G))|$. Since  $G/\Z(G)$ is isomorphic 
to a proper subgroup of $G/\gamma_2(G)$ and $\gamma_2(G)$ is a proper 
subgroup of $\Z(G)$, if follows from Lemma \ref{lemma6} that
\[|\Hom(G/\Z(G), \gamma_2(G))| < |\Hom(G/\gamma_2(G), \Z(G))|.\]
Hence 
\[|\Aut_c(G)| \le \Hom(G/\Z(G), \gamma_2(G))| < |\Hom(G/\gamma_2(G), \Z(G))|
= |\Autcent(G)|.\]
This contradicts the fact that $G$ satisfies Hypothesis A. Hence $\gamma_2(G)
= \Z(G)$. This completes the proof of the proposition. \hfill $\Box$
\end{proof}

Let $G$ be a finite $p$-group of class $2$ such that $\Z(G) \le \Phi(G)$.  Then  ${\bar G} = G/\Z(G)$, being  finite abelian,  admits a minimal basis.  Let $\{{\bar x}_1, \ldots, {\bar x}_d\}$ be a minimal basis for ${\bar G}$, where ${\bar x}_i = x_i\Z(G)$. Now $G/\Phi(G) \cong (G/\Z(G))/(\Phi(G)/\Z(G)) \cong {\bar G}/\Phi({\bar G})$, since $\Z(G) \le \Phi(G)$. Thus the set $\{x_1 \Phi(G), \ldots,  x_d \Phi(G)\}$ minimally generates $G/\Phi(G)$, which implies that the set  $\{x_1, \ldots, x_d\}$ minimally generates $G$.  Now let $x \in G-\Phi(G)$. Therefore $x\Z(G) \in G/\Z(G) - \Phi(G)/\Z(G)$. Thus we can find a minimal basis $\{{\bar x}_1, \ldots, {\bar x}_d\}$ for ${\bar G}$ such that ${\bar x} = x\Z(G) = {\bar x}_i$ for some $1 \le i \le d$.  As a consequence of this discussion, we get the following result.
 
\begin{lemma}\label{lemma8a}
Let $G$ be a finite $p$-group of class $2$ such that $\Z(G) \le \Phi(G)$. Then the following holds true:
\begin{enumerate}
\item Any minimal basis $\{{\bar x}_1, \ldots, {\bar x}_d\}$ for ${\bar G}$ provides  a distinguished minimal generating set $\{x_1, \ldots, x_d\}$ for $G$.
\item Any element $x \in G-\Phi(G)$ can be included in a distinguished minimal generating set for $G$.
\end{enumerate}
\end{lemma}

Since $\Z(G) = \gamma_2(G) \le \Phi(G)$ for a finite $p$-group of class $2$ satisfying Hypothesis A, we readily get
\begin{cor}\label{cor0}
Any finite $p$-group of class $2$ satisfying Hypothesis A, admits a distinguished minimal generating set.
\end{cor}

\begin{prop}\label{prop3}
Let $G$ be a finite $p$-group of class $2$ satisfying Hypothesis A. Let  $\{x_1, \ldots, x_d\}$ be a distinguished minimal generating set for $G$. Then $|\Aut_c(G)| = \Pi_{i =1}^d |x_i^G|$ and $[x_i, G] = \Omega_{m_i}(\gamma_2(G))$, where $p^{m_i}$ is the order of ${\bar x}_i$. Moreover, $[x_1, G] = \gamma_2(G)$.
\end{prop}
\begin{proof}
Let $\{x_1, \ldots, x_d\}$ be a distinguished minimal generating set for $G$ such that  order of ${\bar x}_i$ is $p^{m_i}$ for $1 \le i \le d$. Notice that  $|\Aut_c(G)| \le \Pi_{i=1}^d |x_i^G|$, as there are not more than $|x_i^G|$ choices for the image of $x_i$ under any class preserving 
automorphism of $G$. Since the exponent of the subgroup $[x_i,G]$  is equal to 
the order of ${\bar x}_i = x_i\Z(G) \in G/\Z(G)$  for any  $1 \le i \le d$ (Lemma \ref{lemma6a}), it follows that 
$|\Hom(\gen{{\bar x}_i}, [x_i, G])| = |[x_i,G]|$. By Proposition \ref{prop2} we have $\Z(G) = \gamma_2(G)$. Thus  
\begin{eqnarray}
|\Aut_c(G)| &=& |\Autcent(G)| = |\Hom(G/\gamma_2(G), \Z(G))| 
= |\Hom(G/\Z(G), \gamma_2(G))|\nonumber\\ 
  &=& \Pi_{i=1}^d |\Hom(\gen{{\bar x}_i}, \gamma_2(G))|
\ge \Pi_{i=1}^d |\Hom(\gen{{\bar x}_i}, [x_i,G])|
= \Pi_{i=1}^d |[x_i, G]| \label{eqlemma9}\\
& =& \Pi_{i =1}^d |x_i^G|.\nonumber
\end{eqnarray}
Hence $|\Aut_c(G)| = \Pi_{i=1}^d |x_i^G|$.

 It now follows from \eqref{eqlemma9} that $\Hom(\gen{{\bar x}_i}, \gamma_2(G)) = \Hom(\gen{{\bar x}_i}, [x_i, G])$ for each $1 \le i \le d$. Notice that $\Hom(\gen{{\bar x}_i}, \gamma_2(G)) = \Hom(\gen{{\bar x}_i}, \Omega_{m_i}(\gamma_2(G))) \cong \Omega_{m_i}(\gamma_2(G))$ (Lemma \ref{lemma2.7}).  Also, as mentioned above, $\Hom(\gen{{\bar x}_i}, [x_i, G]) \cong [x_i, G]$. Hence $|[x_i, G]|  = |\Omega_{m_i}(\gamma_2(G))|$.  Since the exponent of $[x_i, G]$ is equal to the order of ${\bar x}_i$, it follows that $[x_i, G] \le \Omega_{m_i}(\gamma_2(G))$ for each $1 \le i \le d$. Hence $[x_i, G] = \Omega_{m_i}(\gamma_2(G))$ for each $1 \le i \le d$. Since the order of ${\bar x}_1$ (which is equal to the exponent of $G/\Z(G)$) is equal to the exponent of $\gamma_2(G)$, $[x_1, G] = \gamma_2(G)$. This  completes the proof of the proposition.         \hfill $\Box$
\end{proof}

In the following corollary we show that the order of $\Aut_c(G)$, obtained in Proposition \ref{prop3}, is independent of the choice of distinguished minimal generating set for $G$.

\begin{cor}\label{cor1}
Let $G$ be a finite $p$-group of class $2$ satisfying Hypothesis A. Let the invariants of $G/\Z(G)$ be $p^{m_1}, \ldots, p^{m_d}$. Then 
$|\Aut_c(G)| = \Pi_{i =1}^d |\Omega_{m_i}(\gamma_2(G))|$. 
\end{cor}
\begin{proof}
Let $p^{m_1}, \ldots, p^{m_d}$ be the  invariants of $G/\Z(G)$. Notice that any distinguished minimal generating set for $G$ can be written (after re-ordering if necessary) as $\{x_1, \ldots, x_d\}$ such that order of ${\bar x}_i$ is $p^{m_i}$ for $1 \le i \le d$. Hence, by Proposition \ref{prop3}, it follows that $|\Aut_c(G)| = \Pi_{i =1}^d |\Omega_{m_i}(\gamma_2(G))|$. This completes the proof.     \hfill  $\Box$
\end{proof}

Some other consequences of Proposition \ref{prop3}  are
\begin{cor}\label{cor2}
Let $G$ be a finite $p$-group of class $2$ satisfying Hypothesis A. Let $x \in G-\Phi(G)$ such that order of $x\Z(G)$ is $p^m$. Then $[x, G] = \Omega_m(\gamma_2(G))$. 
\end{cor}
\begin{proof}
By Lemma \ref{lemma8a}, we can find a distinguished minimal generating set $\{x_1, \ldots, x_d\}$ for $G$ such that $x = x_i$ for some $1 \le i \le d$. Now the assertion follows from Proposition \ref{prop3}. \hfill $\Box$
\end{proof}

\begin{cor}
Let $G$ be a finite $p$-group of class $2$ satisfying Hypothesis A. Let $x \in G-\Phi(G)$ such that order of $x\Z(G)$ is $p^m$. Then $|\C_{G}(x)| = |G/\Z(G)| |\Omega^m(\gamma_2(G))|$.
\end{cor}

We now prove Theorem A.

\begin{thm}[Theorem A]\label{thm1}
Let $G$ be a finite $p$-group of nilpotency class $2$. Then $G$ satisfies
Hypothesis A if and only if $\Z(G) = \gamma_2(G)$ and 
$|\Aut_c(G)| = \Pi_{i =1}^d |\Omega_{m_i}(\gamma_2(G))|$, where $p^{m_1}, \ldots, p^{m_d}$ are the  invariants of $G/\Z(G)$.   
\end{thm}
\begin{proof}
Let $G$ satisfy Hypothesis A. Then by Proposition \ref{prop2}, $\Z(G) = \gamma_2(G)$ and by Corollary \ref{cor1}, 
$|\Aut_c(G)| = \Pi_{i =1}^d |\Omega_{m_i}(\gamma_2(G))|$.

On the other hand, suppose that $\Z(G) = \gamma_2(G)$ and 
$|\Aut_c(G)| =  \Pi_{i =1}^d |\Omega_{m_i}(\gamma_2(G))|$.  Let $\{x_1, \ldots, x_d\}$ be a distinguished minimal generating set  for $G$ such that  order of ${\bar x}_i = x_i\Z(G)$ is $p^{m_i}$ for $1 \le i \le d$.
Since $\Z(G) = \gamma_2(G)$, $G$ is purely non-abelian and therefore
\begin{eqnarray*}
|\Autcent(G)| &=& \Pi_{i=1}^d |\Hom(\gen{{\bar x}_i}, \Z(G))| = \Pi_{i=1}^d |\Hom(\gen{{\bar x}_i}, \Omega_{m_i}(\gamma_2(G))|\\
& =& \Pi_{i =1}^d |\Omega_{m_i}(\gamma_2(G))| = |\Aut_c(G)|.
\end{eqnarray*}
Hence $G$ satisfies hypothesis A. This completes the proof of the theorem. \hfill  $\Box$
\end{proof}

Let $\Z(G) = \Z_1 \times \Z_2 \times \cdots \times \Z_r$ be a cyclic
decomposition of $\Z(G)$ such that $|\Z_1| \ge |\Z_2| \ge \dots \ge |\Z_r| > 1$ and, for each  $i$ such that $1 \le i \le r$, define 
$\Z_i^* =\Z_1 \times \cdots \times \Z_{i-1} \times \Z_{i+1}  \times \cdots \times \Z_r$. 

\begin{prop}\label{prop4}
Let $G$ be a finite $p$-group of class $2$ satisfying Hypothesis A. Let $x \in
G - \Phi(G)$. Then $[x,G] \cap \Z_i \neq 1$ for each $i$ such that $1 \le i \le r$. Moreover, if $|\Z_i|$ is equal to the exponent of $\Z(G)$, then 
$|[x,G] \cap \Z_i|$ is equal to the exponent of  $[x,G]$ and $G/\Z_i^*$ satisfies Hypothesis A.
\end{prop}
\begin{proof}
Since $G$ satisfies Hypothesis A, $\Z(G) = \gamma_2(G)$. We'll make use of this fact without telling so.
Let $x \in G-\Phi(G)$ be such that order of $x\Z(G)$ is $p^m$, where $m \ge 1$ is an integer. Then it follows from Corollary \ref{cor2} that $[x, G] = \Omega_m(\gamma_2(G)) = \Omega_m(\Z(G))$. This proves that $[x,G] \cap \Z_i \neq 1$ for each  $1 \le i \le r$. 

Let $|\Z_i| = p^e$, the exponent of $\Z(G)$. Since the exponent of $[x,G]$ is equal to $|\gen{\bar{x}}|$, where $\bar{x} = x\Z(G)$ (Lemma \ref{lemma6a}), the  exponent of $[x,G]$ is $p^m$.  Thus $e \ge m$. Since $[x, G] = \Omega_m(\Z(G))$, it follows that $[x,G] \cap \Z_i = \Omega_m(\Z_i)$. Hence $|[x,G] \cap \Z_i| = p^m$.

Again let $|\Z_i| = p^e$, the exponent of $\Z(G)$. Assume for a moment that $\Z(G/\Z_i^*) = \Z(G)/\Z_i^*$. Then 
$\Z(G/\Z_i^*) \cong \Z_i$ is cyclic and is equal to $\gamma_2(G/\Z_i^*) =  
\gamma_2(G)/\Z_i^*$, since $\Z_i^* \le \gamma_2(G)$. This implies that
$\Autcent(G/\Z_i^*) = \Inn(G/\Z_i^*) = \Aut_c(G/\Z_i^*)$. Hence   Hypothesys A holds true for $G/\Z_i^*$. Therefore, to complete the proof, it is sufficient to
prove what we have assumed, i.e., $\Z(G/\Z_i^*) = \Z(G)/\Z_i^*$.

Let $x\Z_i^* \in G/\Z_i^* - \Z(G)/\Z_i^*$. Then $x \in G - \Z(G)$. If $x \in G
- \Phi(G)$, then  $[x,G] \not\subseteq \Z_i^*$ and therefore 
$x\Z_i^* \not\in \Z(G/\Z_i^*)$. So let $x \in \Phi(G) - \Z(G)$. Then there
exists an element $y \in G - \Phi(G)$ such that $x = y^jz$ for some positive
integer $j$ and some element $z \in \Z(G)$. Now 
\[[x,G] = [y^jz, G] = [y,G]^j, \;\; 1 \le j < \text{exponent of}\;\; [y,G].\]
Since $y \in G-\Phi(G)$, $|[y,G] \cap \Z_i|$ is equal to the exponent of  $[y,G]$. So it follows that $[y,G]^j \not\subseteq
\Z_i^*$ for any non-zero $j$ which is strictly less than the exponent of
$[y,G]$.  Thus $[x, G] \not\subseteq \Z_i^*$. This implies that $[x\Z_i^*,
G/\Z_i^*] \neq 1$. Hence $x\Z_i^* \not\in \Z(G/\Z_i^*)$. This proves that
$Z(G/\Z_i^*) = \Z(G)/\Z_i^*$, completing  the proof of the proposition. \hfill $\Box$
\end{proof}

For the proof of our next result we need the following important result.
\begin{thm}[\cite{BBC}, Theorem 2.1]\label{thm0a}
Let $G$ be a finite $p$-group of nilpotency class $2$ with cyclic center.  Then $G$ is a central product either of  two generator subgroups with cyclic center  or   two generator subgroups with cyclic center and a cyclic subgroup.
\end{thm}

Now we are ready to prove Theorem B.
\begin{thm}[Theorem B]\label{thm2}
Let $G$ be a finite $p$-group of class $2$ satisfying Hypothesis A. Then $d(G)$ is  even.
\end{thm}
\begin{proof}
Let $|\Z_i|$ be equal to the exponent of $\Z(G)$ and  $G^*$ denote the factor group $G/\Z_i^*$. Then it follows from Proposition \ref{prop4} that $\Z(G^*) = \gamma_2(G^*)$ is cyclic of order 
$|\Z_i|$. Since $G^*$ is purely non-abelian,  it follows from Theorem \ref{thm0a} that $G$ is a central product of  two generator subgroups with cyclic center.  Hence $G^*/\Z(G^*) \cong G/\Z(G)$ is a direct product of even number of cyclic subgroups of  $G^*/\Z(G^*)$. Thus $d(G/\Z(G))$ is even. Since $\Z(G) = \gamma_2(G)$, it follows that $d(G)$ is  even. \hfill $\Box$
\end{proof}

\section{Some further properties and examples}

In this section we discuss some more properties and give some examples of finite groups which satisfy Hypothesis A. We start with the 
following concept of isoclinism of groups,  introduced by P. Hall \cite{pH40}.

Let $X$ be a finite group and $\bar{X} = X/\Z(X)$. 
Then commutation in $X$ gives a well defined map
$a_{X} : \bar{X} \times \bar{X} \mapsto \gamma_{2}(X)$ such that
$a_{X}(x\Z(X), y\Z(X)) = [x,y]$ for $(x,y) \in X \times X$.
Two finite groups $G$ and $H$ are called \emph{isoclinic} if 
there exists an  isomorphism $\phi$ of the factor group
$\bar G = G/\Z(G)$ onto $\bar{H} = H/\Z(H)$, and an isomorphism $\theta$ of
the subgroup $\gamma_{2}(G)$ onto  $\gamma_{2}(H)$
such that the following diagram is commutative
\[
 \begin{CD}
   \bar G \times \bar G  @>a_G>> \gamma_{2}(G)\\
   @V{\phi\times\phi}VV        @VV{\theta}V\\
   \bar H \times \bar H @>a_H>> \gamma_{2}(H).
  \end{CD}
\]
The resulting pair $(\phi, \theta)$ is called an \emph{isoclinism} of $G$ 
onto $H$. Notice that isoclinism is an equivalence relation among finite 
groups.  

Let $G$ be a finite $p$-group. Then it follows from \cite{pH40} that there exists a finite $p$-group $H$ in the isoclinism family of $G$ such that 
$\Z(H) \le \gamma_2(H)$. Such a group $H$ is called a \emph{stem group} in the isoclinism family of $G$.  

The following theorem shows that the group of class preserving automorphisms is independent of the choice of a group in a given isoclinism family of groups.

\begin{thm}[\cite{mYp5}, Theorem 4.1]\label{thm2a}
Let $G$ and $H$ be two finite non-abelian isoclinic groups. Then
$\Aut_c(G) \cong \Aut_c(H)$.
\end{thm}

\begin{prop}\label{prop5}
Let $G$ and $H$ be two non-abelian finite $p$-groups which are isoclinic. Let $G$ satisfy Hypothesis A. Then $H$ satisfies Hypothesis A if and only if $|H| = |G|$.
\end{prop}
\begin{proof}
Suppose that $H$ satisfies Hypothesis A. Then $\gamma_2(H) = \Z(H)$. Since $G$ and $H$ are isoclinic and $G$ satisfies Hypothesis A, it follows that $|\Z(G)| = |\gamma_2(G)| =  |\gamma_2(H)| = |\Z(H)|$ and $|G/\Z(G)| = |H/\Z(H)|$. Hence $|H| = |H/\Z(H)| |\Z(H)| = |G/\Z(G)| |\Z(G)| = |G|$.

Conversely, suppose that $|H| = |G|$. It is easy to show that $\gamma_2(H) = \Z(H)$. Let $p^{m_1}, \ldots, p^{m_d}$ be the invariants of $G/\Z(G) \cong H/\Z(H)$. Since  $\gamma_2(H) \cong \gamma_2(G)$, we have $\Omega_{m_i}(\gamma_2(H)) \cong  \Omega_{m_i}(\gamma_2(G))$.  Since $G$ and $H$ are isoclinic, it follows from Theorem \ref{thm2a} that $\Aut_c(G) \cong \Aut_c(H)$.  Hence  $|\Aut_c(H)| =  \Pi_{i =1}^d |\Omega_{m_i}(\gamma_2(G))| = \Pi_{i =1}^d |\Omega_{m_i}(\gamma_2(H))|$. That $H$ satisfies Hypothesis A, now follows from Theorem \ref{thm1}. \hfill $\Box$
\end{proof}

Let $G$ be a finite group and $1 \neq N$ be a normal subgroup of
$G$. $(G,N)$ is called a \emph{Camina pair} if $xN \subseteq x^G$ for
all $x \in G-N$. A group $G$ is called a \emph{Camina group} if $(G,\gamma_{2}(G))$ 
is a  Camina pair. So if $G$ is a Camina group, then $x\gamma_2(G) \subseteq 
x^G$ all $x \in G-\gamma_2(G)$. This is equivalent to saying  that $\gamma_2(G) 
\subseteq [x, G]$ all $x \in G-\gamma_2(G)$. Since $[x, G] \subseteq
\gamma_2(G)$, it follows that $G$ is a  Camina group if and only if 
$\gamma_2(G)= [x, G]$ all $x \in G-\gamma_2(G)$.

\begin{prop}\label{prop5a}
Let $G$ be a finite special $p$-group. Then Hypothesis A holds true for $G$ if and only if $G$ is a Camina group. 
\end{prop}
\begin{proof}
First suppose that Hypothesis A holds true for $G$. Since $\gamma_2(G) = \Phi(G)$, we only need to show that $[x, G] = \gamma_2(G)$ for all $x \in G -\Phi(G)$. Let $x \in G -\Phi(G)$.  By Corollary \ref{cor2}, we have $[x, G] = \Omega_m(\gamma_2(G))$, where order of $x\Z(G)$ in $G/\Z(G)$ is $p^m$. Since $G$ is special, $m =1$ and the exponent of $\gamma_2(G)$ is $p$. Thus  $[x, G] = \gamma_2(G)$.

Conversely suppose that $G$ is a Camina group. Then it follows from \cite[Theorem 5.4]{mY07} that $|\Aut_c(G)| = |\gamma_2(G)|^d$, where $|G/\Phi(G)| = p^d$. Since $G$  is special, it follows that $\gamma_2(G) = \Z(G)$ and $p^{m_1} =p, p^{m_2} =p , \ldots, p^{m_d} = p $ are the invariants of $G/\Z(G)$. Thus $\gamma_2(G) = \Omega_{m_i}(\gamma_2(G))$.   Now we can use  Theorem \ref{thm1} to deduce  that $G$ satisfies Hypothesis A and the proof is complete. \hfill $\Box$
\end{proof}

Since every finite Camina $p$-group of nilpotency class $2$ is special \cite{iM81}, we readily get  
\begin{cor}
Let $G$ be a finite Camina $p$-group of nilpotency class $2$. Then $G$ satisfies Hypothesis A.
\end{cor}

With this much information we get

\noindent{\bf Example 1.}  All finite Camina $p$-groups of class $2$ satisfy Hypothesis A. In particular, all finite extra-special $p$-groups  satisfy Hypothesis A. Examples of non extra-special Camina $p$-groups of class $2$ can be found in  \cite{DS96} and \cite{iM81}.

Now we construct an example of  a finite $p$-group satisfying Hypothesis A, which is not a Camina group.

\noindent{\bf Example 2.} Let  $R$ be the factor ring $S/p^2S$, where $S$ is the ring of $p$-adic 
integers in the unramified extension of degree 2 over the $p$-adic 
completion $\mathbb Q_p$ of the rational numbers $\mathbb Q$. 
Form the group $G$ of all $3 \times 3$ matrices
\[   M(x,y,z) = \left (  \begin{matrix} 1 & 0 & 0 \\
                                        x & 1 & 0 \\
                                        z & y & 1
                         \end{matrix} \right )            \]
for $x,y,z \in R$. The additive group of $R$ is the direct
sum of two copies of a cyclic group of order $p^2$.  The factor ring
$R/pR$ modulo the ideal $pR$ is a finite field of order $p^2$. Notice 
that commutation in $G$ satisfies
\[[M(x,y,z), M(x',y',z')] = M(0,0,yx' - xy')\]
for any $x,y,z,x',y',z' \in R$.  So both the center $\Z(G)$ and the
derived group $\gamma_2(G)$ consist of all matrices of the form $M(0,0,z)$ 
for $z \in R$.  
Since $M(0,0,z)M(0,0,z') = M(0,0,z+z')$ for all $z, z' \in R$, $\Z(G)$
is noncyclic and equal to $\gamma_2(G)$. Thus the nilpotency class of $G$ is
$2$.  

From the above formula for commutators it follows that
\[[M(x,y,z), G] = M(0,0,Rx+Ry) :=  \{ M(0,0,z) \mid z \in Rx + Ry \}\]
for any $x,y,z \in R$.  Note that there are only three choices for the
ideal $Rx + Ry$ in $R$, namely, $R$, $pR$ and $p^2R = 0$.
Furthermore, all three possibilities happen for suitable $x$ and $y$.
Now
\[[M(1,0,0), G] = [M(p-1,0,0), G] = M(0,0,R) = \Z(G)\]
and
\[[M(1,0,0)M(p-1,0,0), G] = [M(p,0,0), G] = M(0,0,pR) = \Z(G)^p.\]
So
\[[M(1,0,0), G][M(p-1,0,0),G] = \Z(G) > \Z(G)^p = [M(1,0,0)M(p-1,0,0),
G].\] 
This shows that $G$ has a non-central element $x = M(1,0,0)M(p-1,0,0)$ such
that $[x,G] < \gamma_2(G) = \Z(G)$. Hence $G$ is not a Camina group.

Since $\Z(G) = \gamma_2(G)$, $G$ is purely non-abelian. 
Then from Lemma \ref{lemma0}, it follows that for any element $\alpha \in 
\Autcent(G)$ there exists a corresponding element 
$f_{\alpha} \in \Hom(G/\gamma_2(G),\Z(G))$ such that $f_{\alpha}({\bar x}) =
x^{-1} \alpha(x)$ for each ${\bar x} \in G/\gamma_2(G)$.
Let $x,y,z$ be any three elements of $R$, and $i = 0,1,2$ 
be such that $Rx + Ry = p^iR$.  
Then the element $M(x,y,z)\Z(G)$ of $G/\Z(G)$ lies in $(G/\Z(G))^{p^i}$.  
So its image $f(M(x,y,z)\Z(G))$ lies in
\[\Z(G)^{p^i} = M(0,0,p^iR) = M(0,0, Rx + Ry) = [M(x,y,z), G].\]
Thus $f(g) \in [g,G]$ for all $g \in G$, and $\alpha \in \Aut_c(G)$. Since
$\Aut_c(G) \le \Autcent(G)$, this proves that  $\Aut_c(G) = \Autcent(G)$.

\section{Groups $G$ with $\Aut(G) = \Aut_c(G)$}

In this section we prove Theorem C. Throughout the section, $p$ always denotes an odd prime.  We state some important  known results in the following theorem.

\begin{thm}\label{thmo}
The following statements hold true.
\begin{enumerate} 
\item {Let $G$ be a non-abelian  $p$-group of order $p^5$ or less. Then $\Aut(G)$ is non-abelian. \cite[Theorem 5.2]{bE75} }\label{thm3} 
\item {There is no group $G$ of order $p^6$ whose automorphism group is an abelian $p$-group. \cite[Proposition 1.4]{mM94}}\label{thm4} 
\item{Let $G$ be a non-abelian finite $p$-group such that $\Aut(G)$ is abelian. Then $d(G) \ge 4$. \cite{mM95}} \label{thm4a}
\item{Let $G$ be a  non-cyclic finite $p$-group, $p$ odd, for which $\Aut(G)$ is abelian. Then $p^{12}$ divides $|\Aut(G)|$. \cite[Main Theorem]{pH95} } \label{thm5}
\item{Let $G$ be a non-cyclic group of order $p^7$. If $\Aut(G)$ is abelian, then it must be of order $p^{12}$. \cite[Theorem 1]{BY98} }  \label{thm6}
\item {Let $G$ be a finite non-cyclic $p$-group such that $\Aut(G)$ is abelian. Then $\gamma_2(G) \cong \C_{p^{m}} \times  \C_{p^{m}}$ \text{or} $\C_{p^{m}} \times  \C_{p^{m}} \times \C_{p^{m_1}}  \times \cdots \times \C_{p^{m_k}}$, where $p^m$ is the exponent of $\gamma_2(G)$ and $m \ge m_i$ for $1 \le i \le k$. \cite[Lemma 4(12)]{BY98} }  \label{thm6a}
\item{Let $G$ be a finite Camina $p$-group of class $2$ such that $d(G) = n$ and $d(\gamma_2(G)) = m$ for some positive integers $n$ and $m$. Then $n$ is even and $n \ge 2m$. (Follows from 
\cite[Theorems 3.1, 3.2]{iM81})}\label{thm8}
\end{enumerate}
\end{thm}

Now we start the proof of Theorem C.

\begin{lemma}\label{lemma12}
Let $G$ be a non-abelian finite $p$-group such that $\Aut(G) = \Aut_c(G)$, where $p$ is an odd prime. Then $\gamma_2(G)$ cannot be cyclic.
\end{lemma}
\begin{proof}
Suppose that $\gamma_2(G)$ is cyclic. It now follows from \cite[Theorem 3]{yC82} that $\Aut(G) = \Aut_c(G) = \Inn(G)$, which is not possible by a celebrated theorem of W. Gasch\"utz \cite{wG66}.
\hfill $\Box$

\end{proof}

\begin{lemma}\label{lemma13}
Let $G$ be a non-abelian finite $p$-group such that $\Aut(G) = \Aut_c(G)$. Then $\Aut(G)$ is elementary abelian if and only if $G$ is a Camina special $p$-group. 
\end{lemma}
\begin{proof}
Suppose that $\Aut(G) = \Aut_c(G)$ is elementary abelian. Then $G/\Z(G) \cong \Inn(G)$ is elementary abelian. Since   $\Z(G) = \gamma_2(G) \le \Phi(G)$, it follows that $G$ is a special $p$-group. Hence, by Proposition \ref{prop5a}, $G$ is a Camina special $p$-group.

Conversely, suppose that $G$ is a Camina $p$-group of class $2$.  Then $\Inn(G) \cong G/\Z(G)$ is elementary abelian. Hence it follows that $\Aut(G) = \Aut_c(G)$ is elementary abelian. \hfill $\Box$
\end{proof}

\begin{remark}
If $G$ is a special $p$-group and  $\Aut(G) = \Autcent(G)$, then it is not difficult to prove that $\Aut(G)$ is elementary abelian. But the converse is not true.  The examples of  non-special finite $p$-groups $G$ such that $\Aut(G) = \Autcent(G)$ is elementary abelian can be found in \cite{JRY}.
\end{remark}

\begin{prop}\label{prop6}
Let $G$ be a finite $p$-group of nilpotency class $2$ such that $\Aut(G) = \Aut_c(G)$, where $p$ is an odd prime. Then $|G| \ge p^8$.
\end{prop}
\begin{proof}
Since the nilpotency class of $G$ is $2$,  $\Aut(G) = \Aut_c(G)$ is abelian.  It  now follows from Theorem \ref{thmo}\eqref{thm3} and Theorem \ref{thmo}\eqref{thm4}  that $|G| \ge p^7$. Let $|G| = p^7$. Then $|\Aut(G)| = p^{12}$ (by Theorem \ref{thmo}\eqref{thm6}).  Now using the fact that $\Z(G) = \gamma_2(G)$ and $|\Aut(G)| = |\Autcent(G)| = |\Hom(G/\Z(G), \Z(G)| = p^{12}$, it follows from Lemma \ref{lemma12} and Theorems \ref{thm2},  \ref{thmo}\eqref{thm4a} and \ref{thmo}\eqref{thm5}   (by looking various possibilities for the order and structure of $\Z(G)$) that  $G$ is a special $p$-group with $|G/\Z(G)| = p^4$ and $|\Z(G)| = p^3$.  It then follows from Proposition \ref{prop5a} that $G$ is a Camina special $p$-group, which is not possible by Theorem \ref{thmo}\eqref{thm8}.  This completes the proof. \hfill $\Box$
\end{proof}

The following lemma seems basic.
\begin{lemma}\label{lemma14}
Let $G$ be a finite $p$-group of nilpotency class $2$ such that $\gamma_2(G) \cong \C_{p^m} \times \C_{p^m}$ for some positive integer $m$. Then $G/\Z(G) \cong  \C_{p^m} \times \C_{p^m} \times \C_{p^m} \times  H$ for some  abelian group $H$.
\end{lemma}
\begin{proof}
Since the exponent of $\gamma_2(G)$ is $p^m$, we can write $G/\Z(G) =  \gen{{\bar x}_1} \times \gen{{\bar x}_2}  \times  K/\Z(G)$ for some subgroup $K$ of $G$ containing $\Z(G)$ such that $|\gen{[x_1, x_2]}| = |\gen{{\bar x}_1}| = |\gen{{\bar x}_2}| = p^m$.   We only need to prove that the exponent of $K/\Z(G)$ is $p^m$. 
Since  $d(\gamma_2(G)) = 2$, we can find an element $w \in \gamma_2(G)$ such that  $\gamma_2(G) = \gen{[x_1, x_2], w}$, where
\[w =   [x_1, x_2]^{a}[x_1, k]^{a_1}[x_2, k']^{a_2} \Pi_{k_i, k_j \in K}[k_i, k_j]^{b_{ij}}\]
for some $k, k', k_i, k_j \in K$.
Let  $u = [x_1, k_2]^{a_1}[k_1, x_2]^{a_2} \Pi_{k_i, k_j \in K}[k_i, k_j]^{b_{ij}}$. Then notice that $[x_1, x_2]$ and $u$ also generate $\gamma_2(G)$. If the exponent of  $K/\Z(G)$ is less than $p^m$, then order of the element $u$ is also less than $p^m$. Thus $|\gamma_2(G)| < p^{2m}$, which is a contradiction to the given fact that $\gamma_2(G) \cong \C_{p^m} \times \C_{p^m}$. Hence the exponent of $K/\Z(G)$ is $p^m$ and the proof of the lemma is complete. \hfill $\Box$
\end{proof}

\begin{prop}\label{prop7}
Let $G$ be a finite $p$-group such that $\Aut(G) = \Aut_c(G)$ is abelian. Then   $|\Aut(G)| = p^{12}$ if and only if $|G| = p^8$.
\end{prop}
\begin{proof}
Let $|\Aut(G)| = p^{12}$.  By Proposition \ref{prop6}, we can assume that $|G| \ge p^8$. Also, by Theorem \ref{thm2}, $G/Z(G)$ is minimally generated by even number of elements.  Notice that
\[|\Hom(G/\Z(G), \Z(G))| = |\Autcent(G)| = |\Aut(G)| = p^{12}.\]
Since $\Z(G) = \gamma_2(G)$ can not be cyclic, a copy of $\C_p \times \C_p$ is sitting inside $\Z(G)$. Suppose that $d(\Z(G)) \ge 3$. Since $d(G/\Z(G)) \ge 4$ (by Theorem \ref{thmo}\eqref{thm4a}), we have 
$|\Aut(G)| \ge p^{12}$.  Notice that the equality holds only when $d(\Z(G)) = 3$,  $d(G/\Z(G)) = 4$  and the exponent of both $\Z(G)$ and $G/\Z(G)$ is $p$. This implies that $|G| = p^7$, which we are not considering.  Thus $d(\Z(G)) = 2$. Let the exponent of $\Z(G)$ is $p^m$ for some positive integer $m$. Then it follows from Theorem \ref{thmo}\eqref{thm6a}  that  $\Z(G) \cong \C_{p^m} \times \C_{p^m}$. We claim that $m = 1$.  Suppose that $m \ge 2$. If $d(G/\Z(G)) \ge 6$, then notice that $|\Aut(G)| > p^{12}$. Thus by Theorem  \ref{thm2}, $d(G/\Z(G)) = 4$ and by Lemma \ref{lemma14}, $G/\Z(G) \cong \C_{p^m} \times \C_{p^m} \times \C_{p^m} \times \C_{p^r}$ for some $1 \le r \le m$. Now it is easy to show that $|\Aut(G)| > p^{12}$. This contradiction proves our claim, i.e., $m=1$. Now the only choice for $d(G/\Z(G))$ to give $|\Aut(G)| = p^{12}$ is  $6$. Thus $G/\Z(G)$ and $\Z(G)$ are  elementary abelian of order $p^6$  and $p^2$ respectively.  Hence $|G| = p^8$.

Conversely, suppose that $|G| = p^8$ and $\Aut(G) = \Aut_c(G)$ is abelian. Then $|\Z(G)| = |\gamma_2(G)| \le p^4$. If $|\Z(G)| = p^4$, then $|G/\Z(G)| = p^4$ and therefore $G/\Z(G)$ is elementary abelian by Theorem \ref{thmo}\eqref{thm4a}. Hence $\Phi(G) = \Z(G) = \gamma_2(G)$ is elementray abelian. This shows that $G$ is special and therefore Camina special. But this is not possible by  Theorem \ref{thmo}\eqref{thm8}. If $|\Z(G)| = p^3$, then by Theorem \ref{thm2} and Theorem \ref{thmo}\eqref{thm4a}, $G/\Z(G)$ can be written as
\[G/\Z(G) \cong \C_{p^2} \times \C_p \times \C_p \times \C_p\]
and by Lemma \ref{lemma12}, $\Z(G)$ can be written as
\[\Z(G) \cong \C_{p^2} \times \C_p.\]
Thus $|\Aut(G)| = |\Autcent(G)| = |\Hom(G/\Z(G), \Z(G))| = p^9$, which is not possible by Theorem \ref{thmo}\eqref{thm6}. Hence $|\Z(G)| = p^2$. Now $\Z(G)$, being non-cyclic  by Lemma \ref{lemma12},  must be isomorphic to $\C_p \times \C_p$. Since the nilpotency class of $G$ is $2$ and $\Z(G) = \gamma_2(G)$,  $G/\Z(G)$ must be elementary abelian of order $p^6$. Hence $|\Aut(G)| = |\Autcent(G)| = |\Hom(G/\Z(G), \Z(G))| = p^{12}$.  \hfill $\Box$
\end{proof}

It only remains to establish the final thread of Theorem C.  For an odd prime $p$, consider the group

\begin{eqnarray}\label{eqnl}
G &=& \langle x_1,\; x_2,\; x_3,\;x_4,\; x_5,\; x_6 \mid x_1^{p^2}= x_2^{p^2}= x_3^p = x_4^p = x_5^p  = x_6^p = 1,\\
 & &[x_1,x_2] = x_1^p,\; [x_1,x_3] = x_2^p, \; [x_2, x_3] = x_1^p, \; [x_1, x_4] = x_2^p,\; [x_2,x_4] = x_2^p,\nonumber\\
& & [x_3,x_4] = x_2^p,\; [x_1,x_5] = x_2^p,\; [x_2,x_5] = x_1^p,\; [x_3,x_5] = x_2^p, \; [x_4,x_5] = x_1^p,\nonumber\\
& & [x_1,x_6] = x_2^p,\; [x_2,x_6] = x_2^p,\; [x_3,x_6] = x_1^p, \; [x_4,x_6] = x_1^p,\; [x_5,x_6] = x_2^p \rangle \nonumber
\end{eqnarray}

The following lemma completes the proof of Theorem C.
\begin{lemma}
The group $G$, defined in \eqref{eqnl}, is a  special $p$-group of order $p^8$ with $|\Z(G)| = p^2$ for all odd primes $p$. For $p = 3$,   $G$ is a Camina group and $\Aut(G) = \Aut_c(G)$ is elementary abelian of order $p^{12}$.
\end{lemma}
\begin{proof}
It is fairly easy to show that $G$ is a  special $p$-group of order $p^8$ with $|\Z(G)| = p^2$. Then by Theorem \ref{lemma0}, it follows that $|\Autcent(G)| = p^{12}$.   For $p = 3$,  using GAP \cite{gap} one can easily establish (i) $G$ has $737$ conjugacy classes and therefore it is a Camina group;  (ii)  $\Aut(G)$ is elementary abelian of order $p^{12}$.  Thus, by Proposition \ref{prop5a}, $G$ satisfies Hypothesis A. Hence $|\Aut(G)| = |\Aut_c(G)| = p^{12}$.  \hfill $\Box$
\end{proof}

We conclude this section with some questions. A finite abelian group $Y$ of exponent $e$ is said to be \emph{homocyclic}  if the set of its invariants is $\{e\}$. \\

\noindent{\bf Question 1.}  Let $G$ be a finite $p$-group  of class $2$ such that $\Aut(G) = \Aut_c(G)$. Is it true that $d(G) \ge 2 d(\gamma_2(G))$? If the answer is negative, what is the relationship between $d(G)$ and $d(\gamma_2(G))$?\\

\noindent{\bf Question 2.}  Let $G$ be a finite $p$-group  of class $2$ such that $\Aut(G) = \Aut_c(G)$ and $G/\Z(G)$ is homocyclic. Is  $\gamma_2(G)$ homocyclic? If not, how big homocyclic group of the highest exponent, $\gamma_2(G)$ contains?\\

\noindent{\bf Acknowledgements.} I thank Prof. Everett C. Dade for some useful e-mail discussion during the year 2008. Example 2 above is due to him.

\end{document}